\documentclass[11pt,a4paper]{amsart}

\usepackage{amsmath,amsthm,amssymb}

\makeatletter
 
  \@addtoreset{equation}{section}
\makeatother

\newtheorem{thm}{Theorem}

\newtheorem{lem}[thm]{Lemma}
\newtheorem{prop}[thm]{Proposition}
\newtheorem{cor}[thm]{Corollary}

\newcommand{\beq}[1]{\begin{equation}\label{#1}}
\newcommand{\eeq}{\end{equation}}
\newcommand{\bep}{\begin{proof}}
\newcommand{\eep}{\end{proof}}
\newcommand{\bec}[1]{\begin{cor}\label{#1}}
\newcommand{\eec}{\end{cor}}
\newcommand{\bepr}[1]{\begin{prop}\label{#1}}
\newcommand{\eepr}{\end{prop}}
\newcommand{\bel}[1]{\begin{lem}\label{#1}}
\newcommand{\eel}{\end{lem}}
\newcommand{\bet}[1]{\begin{theorem}\label{#1}}
\newcommand{\eet}{\end{theorem}}

\newcommand{\bZ}{\ensuremath{\mathbb{Z}}}
\newcommand{\bQ}{\ensuremath{\mathbb{Q}}}
\newcommand{\bR}{\ensuremath{\mathbb{R}}}

\let\eps=\varepsilon

\def\hh{{\rm h}\kern.4pt}

\begin{document}

\title[Two-parameter families of Diophantine triples]{Two-parameter families of uniquely extendable Diophantine triples}{}

\author[M. Cipu]{Mihai Cipu}
\address{Simion Stoilow Institute of Mathematics of the 
Romanian Academy, Research unit nr. 5,
P.O. Box 1-764, RO-014700 Bucharest, Romania}
\email{Mihai.Cipu@imar.ro}

\author[Y. Fujita]{Yasutsugu Fujita}
\address{Department of Mathematics, College of Industrial Technology, 
Nihon University, 2-11-1 Shin-ei, Narashino, Chiba, Japan}
\email{fujita.yasutsugu@nihon-u.ac.jp}

\thanks{The second author is partially supported by JSPS KAKENHI Grant Number 16K05079.}

\author[M. Mignotte]{Maurice Mignotte}
\address{D\'epartement de Math\'ematique, Universit\'e de Strasbourg, 
67084 Strasbourg, France}
\email{mignotte@math.u-strasbg.fr}

\title[Two-parameter families of Diophantine triples]{Two-parameter families of uniquely extendable Diophantine triples}{}


\begin{abstract}
 Let $A,\,K$ be positive integers and  $\varepsilon \in \{-2,-1,1,2\}$. 
The main contribution of the paper is a proof that each 
of the $D(\varepsilon^2)$-triples
$\{K,A^2K+2\varepsilon A,(A+1)^2K+2\varepsilon(A+1)\}$   has 
unique extension to a $D(\varepsilon^2)$-quadruple. This is used 
to slightly strengthen the conditions required for the existence
of a $D(1)$-quintuple whose smallest three elements form a regular
triple.
\end{abstract}

\maketitle


\section{Introduction}\label{sec:intr}
Let $n$ be an arbitrary integer. A set of positive integers is
called $D(n)$-tuple if the product of any two distinct elements
increased by $n$ is a perfect square. In case the set has cardinality
2 (3, 4 or 5) one speaks of a $D(n)$-pair (triple, quadruple or
quintuple, respectively).

Among $D(n)$-sets, the most studied ones are  those with $n=1$.
The interest and efforts are driven towards confirmation of the 
folklore conjecture that predicts there are no $D(1)$-quintuples.
A good deal of necessary conditions for the existence of a   
$D(1)$-quintuple is presently known. In a recent work on this 
subject~\cite{CFF} it is shown that if $\{a,b,c,d,e\}$ is a  
$D(1)$-quintuple with $a<b<c<d<e$ and $c=a+b+2\sqrt{ab+1}$ then 
$b<a^3$. Therefore, the positive integer $r$ satisfying $ab+1=r^2$
is less than $a^2$. In the extremal case $r=a^2-1$ the three
smallest elements of such a $D(1)$-quintuple are $a$, $b=a^3-2a$,
$c=a(a+1)^2-2(a+1)$. One of the present authors has remarked that
this triple is formally obtained by specializing $k$ to $-a$ in
the triple $\{k, a^2k+2a,(a+1)^2k+2(a+1)\}$ considered in~\cite{HT1} 
and then changing the sign of all entries. To put it differently, 
our triple appears in the two-parameter family 
$\{K,A^2K-2 A, (A+1)^2K-2(A+1)\}$ dual to that considered by He and
Togb\'e.  A closer look at~\cite{HT1} reveals that the companion
$D(1)$-triple is in fact mentioned in the introduction to that
paper without further study.

A close similarity of results on $D(1)$- and $D(4)$-sets
is well documented in literature, 
as found, e.g., by comparing \cite{dus} and \cite{Fany} with 
\cite{Filipin-2,Filipin-4} and \cite{Filipin-3}. 
One of the common properties is that any $D(\sigma)$-triple with 
$\sigma \in \{1,4\}$ can be extended to a $D(\sigma)$-quadruple. 
More precisely,  if $\{a,b,c\}$ is a $D(\sigma)$-triple, then 
$\{a,b,c,d_+\}$ is a $D(\sigma)$-quadruple, where
\[
d_+=a+b+c+\frac{2}{\sigma}(abc+rst)
\]
and 
\[
r=\sqrt{ab+\sigma},\quad s=\sqrt{ac+\sigma},\quad t=\sqrt{bc+\sigma}.
\]
Such a $D(\sigma)$-quadruple is called {\it regular}, and it is 
conjectured that any $D(\sigma)$-quadruple is regular (cf.~\cite{ahs,DR}). 
Among $d$'s such that $\{a,b,c,d\}$ is a $D(\sigma)$-quadruple with 
$a<b<c<d$, the smallest integer is known to be $d_+$ from 
\cite[Proposition 1]{dus} and \cite[Proposition 1]{Filipin-3}.

The present paper deals with two closely related families, viz.~those 
of $D(4)$-triples mentioned
in the abstract. The outcome of our study is the theorem below, showing 
that each of the triples under scrutiny has unique extension to
quadruple. In particular, the next result shows that the conjecture
mentioned above is true for the families examined in this paper.

\begin{thm}\label{thm:main}
Let $A,\,K$ be positive integers. 
If $\{K,A^2K+2\varepsilon A,(A+1)^2K+2\varepsilon(A+1),d\}$ is a 
$D(\varepsilon^2)$-quadruple with $\varepsilon \in \{-2, -1, 1, 2\}$, 
then they are regular, in other words, we have
\begin{align}\label{d}
d= d_+=\varepsilon^{-2}(2A^2+2A)^2K^3+\varepsilon^{-1}(16A^3+24A^2+8A)K^2+
(20A^2+20A+4)K+\varepsilon(8A+4).
\end{align}
\end{thm}

Note that the assumption in Theorem \ref{thm:main} immediately implies 
that $d$ is the largest element in the quadruple. 
Indeed, substituting $a=K$, $b=A^2K+2\varepsilon A$ and 
$c=(A+1)^2K+2\varepsilon (A+1)$ shows
\begin{align}\label{c}
c=a+b+2r
\end{align}
with $r=\sqrt{ab+\varepsilon^2}$, 
and if $d<c$, then one can deduce from the minimality of ``$d_+$'' 
mentioned above that
\begin{align}\label{cge}
c \ge a+b+d+\frac{2}{\varepsilon^2}(abd+rs't')
\end{align}
with $s'=\sqrt{ad+\varepsilon^2}$ and $t'=\sqrt{bd+\varepsilon^2}$. 
It follows from (\ref{c}) and (\ref{cge}) that $d \le 0$, a contradiction. 

It is also to be noted that it suffices to prove the thesis for $\varepsilon$
even. Indeed, if $\varepsilon= \pm 2$ and $K$ is even then 
simplification by $2$ results in $D(1)$-triples belonging to 
the desired families and transforms the fourth element $d$ in the 
required form. Conversely, doubling all the entries of a $D(1)$-triple
in the indicated families, one obtains a $D(4)$-triple in the 
families with doubled $\varepsilon$.

The result published in~\cite{HT2} for the $\varepsilon=1$ case
says that the conclusion of our Theorem~\ref{thm:main} holds 
for either $A \le 10$ or $A \ge 52330$.  Similar results have 
been published in~\cite{fht} for $\varepsilon=2$.  More precisely, 
the statement has been proved for  $A \le 22$ as well as for 
$A \ge 51767$.

Theorem \ref{thm:main} has the following corollary on extendability
of more general $D(\varepsilon^2)$-triples $\{a,b,a+b+2r\}$, where 
$r=\sqrt{ab+\varepsilon^2}$, to quadruples. 
\begin{cor}\label{cor:main}
Let $\varepsilon \in \{-2,-1,1,2\}$. 
Let $\{a,b,c,d\}$ be a $D(\varepsilon^2)$-quadruple with $a<b<c$ 
and $c=a+b+2r$, where $r=\sqrt{ab+\varepsilon^2}$. 
If $r \equiv \varepsilon \pmod{a}$, then $d=d_+$. In particular, if 
$a$ has either of the forms $4|\varepsilon|$, $p^e$ and $2p^e$ with 
$p$ an odd prime and $e$ a non-negative integer, then $d=d_+$. 
\end{cor}

The progress achieved in our work is largely due to a version of
Rickert's theorem tailored for the triples we study. The novelty
in its proof (given in Section~\ref{R}) is to exploit,  besides $N$  
being  divisible by $A$ (where $ N=(A^2+A)K/2 \pm 2A $), the fortunate 
fact that both $N \mp 2A$ and $N \pm 2$ are divisible by $A+1$.                                      
Theorem~\ref{thm:RB} in conjunction with an older theorem of
Laurent~\cite{lau} providing sharp upper bounds for linear forms
in the logarithms of two algebraic numbers allows us to obtain
remarkably small absolute bounds on $A$. Section~\ref{sec:baker2}
contains the details. With some computer help, we next show in
Proposition~\ref{prop:hg2} that if any $D(4)$-triple would be
extendable to two quadruples then $K < 240.24 (A+1)+K_0$ as
soon as $A\ge A_0$. Here,  $A_0$, $K_0$ are small positive integers
determined by a gp script. Such a result is very helpful in
reducing the number of pairs $(A,K)$ for which an application
of Baker-Davenport reduction is required.

In the final section of the paper we come back to the original
problem on $D(1)$-quintuples and slightly improve the bounds
on entries if the smallest ones form a regular triple.

\begin{prop} \label{pr:appfin}
Let  $\{a,b,c,d,e\}$ be a $D(1)$-quintuple  with $a<b<c<d<e$  and  
$c=a+b+2\sqrt{ab+1}$. Then 
$b \le a^3-2a\left\lceil \sqrt{3a+1} \, \right \rceil +3$
and $a\ge 32$.
\end{prop} 

\bepr{pr:appsup}
Let  $\{a,b,c,d,e\}$ be a $D(1)$-quintuple  with $a<b<c<d<e$ and $b < 4a$. Then 
$b \le 4a-4\left\lceil \sqrt{3a+1} \, \right \rceil +3$
and $a\ge 32815$.
\eepr
 
\section{Optimization of Rickert's theorem}\label{R}

The goal of this section is to provide the main technical tool
used in our proof of Theorem~\ref{thm:main}. As already mentioned,
it is a variant of Rickert's theorem that takes into account all 
peculiarities of the families we study.

\begin{thm}\label{thm:RB}
Let $\varepsilon \in \{-2,-1,1,2\}$ and let $A$, $K$ be integers 
satisfying $K \geq 30.03|\varepsilon|^3(A+1)$ with either 
$A \ge 3$ or $A=|\varepsilon|=2$. Put $N=(A^2+A)K/2+\varepsilon A$. 
Then the numbers $\theta_1=\sqrt{1-\varepsilon A/N}$ and 
$\theta_2=\sqrt{1+\varepsilon/N}$ satisfy
\[
\max \left\{\left|\theta_1-\frac{p_1}{q}\right|, 
\left|\theta_2-\frac{p_2}{q}\right|\right\}> 
\left(2.838\cdot10^{28}(A+1)N\right)^{-1}q^{-\lambda}
\]
for all integers $p_1,\,p_2,\,q$ with $q>0$, where
\[
\lambda=1+\frac{\log(20(A+1)N)}{\log \left(\frac{1.338N^2}{|\varepsilon|^3A(A+1)}\right)}<2.
\]
\end{thm}
\begin{proof}
Note that the assumptions $A \geq 3$, $K \geq 30.03|\varepsilon|^3(A+1)$ 
immediately imply $\lambda<2$. The same bound on $\lambda$ is valid
under the hypothesis $A=|\varepsilon|=2$.  

Our task is reduced to finding those 
real numbers satisfying the conditions in the following lemma. 
\begin{lem}\textup{(cf.~\cite[Lemma 3.1]{B})}\label{lem:RB}
Let $\theta_1,\theta_2$ be arbitrary real numbers and $\theta_0=1$. 
Assume that there exist positive real numbers $l,\,p,\,L$ and $P$ with $L>1$ 
such that for each positive integer $k$, we can find integers $p_{ijk}$ 
$(0 \leq i,j \leq 2)$ with nonzero determinant, 
\[
|p_{ijk}| \leq pP^k \enskip (0 \leq i,j \leq 2)
\]
and
\[
\left|\sum_{j=0}^2 p_{ijk}\theta_j \right| \leq l L^{-k}\enskip (0 \leq i \leq 2).
\]
Then
\[
\max\left\{\left|\theta_1-\frac{p_1}{q}\right|,
\left|\theta_2-\frac{p_2}{q}\right|\right\}>C q^{-\lambda}
\]
holds for all integers $p_1,p_2,q$ with $q>0$, where
\[
\lambda=1+\frac{\log P}{\log L}\quad {\it and}\quad
C^{-1}=4pP \left(\max\{1,2l\}\right)^{\lambda-1}.
\]
\end{lem}
Consider the contour integral
\[
I_i(x)=\frac{1}{2\pi\sqrt{-1}}\int_{\gamma}\frac{(1+zx)^k(1+zx)^{1/2}}{(z-a_i)(F(z))^k}dz
\]
for $0 \leq i \leq 2$ and a positive integer $k$, where $a_0,\,a_1,\,a_2$ 
are distinct integers with $a_j=0$ for some $j$, $F(z)=(z-a_0)(z-a_1)(z-a_2)$ 
and $\gamma$ is a closed, counter-clockwise contour enclosing $a_0,a_1,a_2$. 
The integral can be expressed as
\[
I_i(x)=\sum_{j=0}^2p_{ij}(x)(1+a_jx)^{1/2}
\]
for $0 \leq i \leq 2$ with $p_{ij}(x) \in \bQ[x]$ of degree at most $k$ 
(cf.~\cite{R}). From the arguments following Lemma 3.1 in \cite{B} we see that 
\begin{align}\label{lL}
\left|\sum_{j=0}^2p_{ij}(1/N)\left(1+\frac{a_j}{N}\right)^{1/2}\right|<
\frac{27}{64}\left(1-\frac{A}{N}\right)^{-1}
\left\{\frac{27}{4}\left(1-\frac{A}{N}\right)^2N^3\right\}^{-k}
\end{align}
and
\begin{align}\label{pij1}
\left|p_{ij}(1/N)\right|\left(1+\frac{a_j}{N}\right)^{1/2} \leq 
\max_{z \in \Gamma_j}\frac{|1+z/N|^{k+1/2}}{|F(z)|^k},
\end{align}
where the contours $\Gamma_j$ are defined by 
\[
|z-a_j|=\min_{i \neq j}\left\{\frac{|a_j-a_i|}{2}\right\}.
\]
We now take $a_0=-\varepsilon A$, $a_1=0$, $a_2=\varepsilon$. 
Comparing the values of the right-hand side of (\ref{pij1}) in the 
twelve cases for $j \in \{0,1,2\}$ with $\varepsilon \in \{-2,-1,1,2\}$ 
shows that
\begin{align}\label{pij2}
|p_{ij}(1/N)| \leq \left(1+\frac{|\varepsilon|}{2(N+|\varepsilon|)}\right)^{1/2}
\left(\frac{8(1+3|\varepsilon|/(2N))}{|\varepsilon|^3(2A+1)}\right)^k
\end{align}
for all $j$. 
Moreover, the proof of Lemma 3.3 in \cite{R} enables us to write
\[
p_{ij}(1/N)=\sum_{ij}\left(\begin{array}{c}k+\frac{1}{2}\\ h_j \end{array}\right)
C_{ij}^{-1}\prod_{l \neq j}\left(\begin{array}{c}-k_{il}\\ h_l \end{array}\right),
\]
where 
\[
C_{ij}=\frac{N^k}{(N+a_j)^{k-h_j}}\prod_{l \neq j}(a_j-a_l)^{k_{il}+h_l}, 
\]
$k_{il}=k+\delta_{il}$ with $\delta_{il}$ the Kronecker delta, $\sum_{ij}$ 
denotes the sum over all non-negative integers $h_0,\,h_1,\,h_2$ satisfying 
$h_0+h_1+h_2=k_{ij}-1$, 
and $\prod_{l \neq j}$ denotes the product from $l=0$ to $l=2$ omitting $l=j$. 
Let $N=(A^2+A)K/2+\varepsilon A$. 
If $j=0$, then 
\[
|C_{i0}|=\frac{2^{k-h_0}N^kA^{k_{i1}+h_0+h_1-k}(A+1)^{k_{i2}+h_0+h_2-k}|\varepsilon|^{k_{i1}+k_{i2}+h_1+h_2}}{K^{k-h_0}}. 
\]
Thus we have $2^k|\varepsilon|^{3k}A^k(A+1)^kN^kC_{i0}^{-1}\in \bZ$. 
If $j=1$, then 
\[
|C_{i1}|=\frac{2^{k-h_1}N^kA^{k_{i0}+h_0+h_1-k}|\varepsilon|^{k_{i0}+k_{i2}+h_0+h_2}}{\left\{(A+1)K+2\varepsilon\right\}^{k-h_1}}, 
\]
which implies $2^k|\varepsilon|^{3k}A^kN^k C_{i1}^{-1}\in \bZ$. 
If $j=2$, then 
\[
|C_{i2}|=\frac{2^{k-h_2}N^k(A+1)^{k_{i0}+h_0+h_2-k}|\varepsilon|^{k_{i0}+k_{i1}+h_0+h_1}}{(AK+2\varepsilon)^{k-h_2}},
\]
which yields $2^k|\varepsilon|^{3k}A^k(A+1)^kN^kC_{i2}^{-1}\in \bZ$. 
Since 
\[
2^{2k-1}\left(\begin{array}{c}k+\frac{1}{2}\\ h_j \end{array}\right)\in \bZ
\]
for all $j$ (see the proof of Lemma 4.3 in \cite{R}), it is deduced from 
the proof of Theorem 2.5 in \cite{CF} that
\begin{align*}
p_{ijk}:=2^{-1}\{8|\varepsilon|^{3}A(A+1)N\}^k\,\Pi_2(k)^{-1}p_{ij}(1/N) \in \bZ,
\end{align*}
where $\Pi_2(k)$ is an integer satisfying $\Pi_2(k)>1.6^k/(4.09\cdot 10^{13})$. 
It follows from (\ref{lL}), (\ref{pij2}) with the assumptions $A \geq 3$, 
$K \geq 30.03|\varepsilon|^3(A+1)$ that 
\begin{align}\label{in:pPlL}
|p_{ijk}|<pP^k,\quad 
\left|\sum_{j=0}^2p_{ijk}\left(1+\frac{a_j}{N}\right)^{1/2}\right|<lL^{-k},
\end{align}
where
\begin{align*}
p&=2.045\cdot10^{13}\left(1+\frac{|\varepsilon|}{2(N+|\varepsilon|)}\right)^{1/2}<2.047\cdot10^{13},\\
P&=\frac{40A(A+1)N(1+3/(2N))}{2A+1}
<20(A+1)N,\\
l&=2.045\cdot10^{13}\cdot\frac{27}{64}\left(1-\frac{A}{N}\right)^{-1}<8.664\cdot10^{12},\\
L&=\frac{1.35}{|\varepsilon|^{3}A(A+1)}\left(1-\frac{A}{N}\right)^2N^2>
\frac{1.338N^2}{|\varepsilon|^{3}A(A+1)}.
\end{align*}
Inequality (\ref{in:pPlL}) with the above estimates on $p,\,P,\,l,\,L$ holds also for the case $A=|\varepsilon|=2$. 
Therefore, we may take $\lambda$  in Lemma \ref{lem:RB} as in the assertion
of Theorem \ref{thm:RB}, and 
\begin{align*}
C^{-1}&<4\cdot 2.047\cdot10^{13}\cdot20(A+1)N
\left(2\cdot 8.664\cdot10^{12}\right)^{\lambda-1}\\
      &<2.838\cdot10^{28}(A+1)N.
\end{align*}
This completes the proof of the theorem. 
\end{proof}

\section{Auxiliary results for $\eps= -2$}\label{sec2}

For an arbitrary $D(4)$-quadruple $\{a,b,c,d\}$ there exist positive
integers verifying $ab+4=r^2$, $ac+4=s^2$, $bc+4=t^2$, $ad+4=x^2$, 
$bd+4=y^2$, $cd+4=z^2$. Elimination of $d$ yields a system of
generalized Pell equations
\beq{eq:7}
az^2-cx^2=4(a-c),
\eeq 
\beq{eq:8}
bz^2-cy^2=4(b-c).
\eeq 
By well-known structure theorem for solutions of such an equation,
there exist fundamental solutions $(x_0,z_0)$ and $(y_1,z_1)$ 
of~\eqref{eq:7} and~\eqref{eq:8}, respectively, such that 
$z=v_m=w_n$, where
\[
 v_0=z_0, \quad v_1=\tfrac{1}{2}(sz_0+cx_0), \quad v_{m+2}=sv_{m+1}-v_{m},
\]
\[
 w_0=z_1, \quad w_1=\tfrac{1}{2}(tz_1+cy_1), \quad w_{n+2}=tw_{n+1}-w_{n},
\]
and $\vert z_0 \vert <a^{-1/4} c^{3/4}$, $\vert z_1 \vert <b^{-1/4} c^{3/4}$.

The initial terms of these recurrent sequences are severely restricted.

\bel{lefil1} (\cite[Lemma 9]{Filipin-1})
Suppose  the equation $v_{m}=w_{n}$ holds for some nonnegative
integers $m$ and $n$.

\emph{(a)} If both $m$ and $n$ are even then $z_0=z_1$ and 
$\vert z_0 \vert =2$ or $\vert z_0 \vert =(cr-st)/2$ or
$\vert z_0 \vert < 1.608 a^{-5/14} c^{9/14}$.

\emph{(b)} If $m$ is odd and $n$ is even then $\vert z_0 \vert =t$,
$\vert z_1 \vert =(cr-st)/2$, and $z_0z_1 <0$.

\emph{(c)} If $m$ is even and $n$ is odd then $\vert z_1 \vert =s$,
$\vert z_0 \vert =(cr-st)/2$, and $z_0z_1 <0$. 

\emph{(d)} If both $m$ and $n$ are odd then $\vert z_1 \vert =s$,
$\vert z_0 \vert =t$, and $z_0z_1 >0$. 
\eel

We first note that the relationship between the two families
of $D(4)$-triples mentioned in Introduction is more than formal.

\bel{lemin1}
If $K$ is a divisor of $4$ then one has 
\[
\left( K, A^2K-4A,(A+1)^2K-4(A+1)\right)
 = \left( K,B^2K+4B,(B+1)^2K+4(B+1)\right)
\]
for $B=A-4/K$.
\eel

Lemma \ref{lemin1} allows us to assume either $K=3$ or $K \ge 5$, 
since the triples $\{K,A^2K+4A,(A+1)^2K+4(K+1)\}$ will be studied 
in the next section. Moreover, we may assume $A \ge 2$, since the 
family of $D(4)$-triples $\{K,K+4,4K+8\}$ is known to be uniquely 
extendable by~\cite{Fujita-3}.

Throughout this section we denote $a=K$, $b=A^2K-4A$, $c=(A+1)^2K-4(A+1)$,
$r=AK-2$, $s=(A+1)K-2$, $t=A(A+1)K-(4A+2)$. Note that one has $c=a+b+2r$,
which means that the triple $\{a,b,c\}$ is regular. It is equally easy
to check that the element $d$ given by~\eqref{d} coincides with 
$d_+:=a+b+c+2abc+2rst$, so that the quadruple $\{a,b,c,d\}$ is regular.

In the case we are interested in, more precise information on initial 
terms can be obtained.

\bel{leconi}
Suppose $\left( K, A^2K-4A,(A+1)^2K-4(A+1),d\right)$ is a
$D(4)$-quadruple, where $K$, $A$ are integers with $A \ge 2$ and
$K \ge 3$.  
Then any positive solution to the associated system 
of Pell equations satisfies $z=v_{2m}=w_{2n}$, with $x_0=y_1=2$
and $z_0=z_1=\pm 2$.
\eel \bep 
Assuming item (b)  or (d) of Lemma~\ref{lefil1} applies, it results
$bc <t^2 =z_0^2 <a^{-1/2} c^{3/2}$, whence $A^2 K <5A+1$, an
inequality which is incompatible with $A \ge 2$ and $K \ge 3$. 
If item (c) holds then one concludes that one has $a <b^{-1/2} c^{1/2}$, 
equivalently $A^2K^3-4AK^2-(A+1)^2K+4(A+1)<0$, which is false 
for parameters in the ranges $A \ge 2$ and $K \ge 3$.

So possibility (a) occurs. Since for the particular triple we are
studying one has $cr-st=4$, it remains to show that
one cannot have $\vert z_0 \vert < 1.608 a^{-5/14} c^{9/14}$.
Assuming the contrary, it results $\{a,(z_0^2-4)/c,b,c\}$ is a
$D(4)$-quadruple to which Proposition~1 in~\cite{Filipin-3}
applies, giving $ c> \min \{ 0.173 b^{13/2} a^{11/2}, 
0.087 b^{7/2} a^{5/2}\}$, which is obviously false.
\eep

By Lemma \ref{leconi}, one can express any solution to Pellian equation (\ref{eq:8}) as $y=u_n'$, where
\begin{align}\label{cong:y}
u_0'=2,\quad u_1'=t \pm b,\quad u_{n+2}'=tu_{n+1}'-u_n'.
\end{align}
Any solution to the other Pellian equation 
\begin{align}\label{P:ab}
ay^2-bx^2=4(a-b)
\end{align}
deduced from (\ref{eq:7}) and (\ref{eq:8}) is given by $y=u_l''$, where
\begin{align}\label{cong:y'}
u_0''=y_2,\quad u_1''=\dfrac12(ry_2+bx_2),\quad u_{l+2}''=ru_{l+1}''-u_l''
\end{align}
with a solution $(y_2,x_2)$ to (\ref{P:ab}) satisfying
\begin{align}\label{fs:xyz}
|y_2|<\sqrt{\frac{b\sqrt{b}}{\sqrt{a}}}\quad \text{and}\quad 1 \le x_2 <\sqrt{b}.
\end{align}
Considering (\ref{cong:y}) and (\ref{cong:y'}) modulo $b$, 
we see that if $u_{2n}'=u_{2l}''$ has a solution, then $y_2 \equiv 2 \pmod{b}$, 
which together with (\ref{fs:xyz}) implies $y_2=2$ and $x_2=2$. 
Suppose that $u_{2n}'=u_{2l+1}''$ has a solution. 
Then, as seen in \cite[Section 5]{fht}, we have 
\begin{align}\label{eq:br}
bx_2-r|y_2|=4
\end{align}
and $bx_2+r|y_2|<2bx_2<2b\sqrt{b}$.
If $A \ge 3$ and $K \ge 3$, then $b \ge 9a-12 \ge 15$, which together with (\ref{fs:xyz}) yields
\begin{align*}
(bx_2-r|y_2|)(bx_2+r|y_2|)=4b(b-a)-4y_2^2>\frac{4b(8b-3\sqrt{3b}-12)}{9}.
\end{align*}
Hence, we obtain
\[
bx_2-r|y_2|>\frac{2(8b-3\sqrt{3b}-12)}{9\sqrt{b}}>5,
\]
which contradicts (\ref{eq:br}). 
Similarly, in case $A=2$ and $K \ge 6$, we will arrive at a contradiction. 
We have thus showed the following.
\bel{lem:xy2}
Suppose $\left( K, A^2K-4A,(A+1)^2K-4(A+1),d\right)$ is a
$D(4)$-quadruple, where $K$, $A$ are integers with either $A \ge 3$ and $K \ge 3$ or $A=2$ and $K \ge 6$.
Then any positive solution to the associated system 
of Pell equations satisfies $y=u_{2n}'=u_{2l}''$, with $x_2=y_2=2$.
\eel

Throughout the rest of this section, suppose that either of the following holds:
\begin{itemize}
\item $A \ge 3$ and either $K=3$ or $K \ge 5$;
\item $A =   2$ and $K \ge 6$.
\end{itemize}
Lemmas \ref{leconi} and \ref{lem:xy2} enable us to express any solution to the 
system of Pellian equations (\ref{eq:7}) and (\ref{P:ab}) as $x=W_{2m}=V_{2l}$, where
\begin{align*}
W_0=2,\quad & W_1=s\pm a,\quad W_{m+2}=sW_{m+1}-W_m,\\
V_0=2,\quad & V_1=r+a,\quad V_{l+2}=rV_{l+1}-V_l.
\end{align*}

Put  
\[
\alpha =\frac{s+\sqrt{ac}}{2}, \quad \beta =\frac{r+\sqrt{ab}}{2},
\quad \chi =\frac{\sqrt{bc} + \sqrt{ac}}{\sqrt{bc} \pm \sqrt{ab}}. 
\]
Then, in a fashion similar to Lemma 10 in \cite{fht}, one finds 
that if $m \ge 1$, then 
\begin{align}\label{in:L}
0<\Lambda<\alpha^{1-4m},
\end{align}
where $\Lambda=2l\log \beta-2m\log \alpha+\log \chi$. 

\bel{le1}
$\alpha -\beta >K=s-r.$
\eel
\bep This is equivalent to $\sqrt{s^2-4} > K+\sqrt{r^2-4}$.
Squaring this, one arrives at the obvious inequality
$r> \sqrt{r^2-4}$.
\eep

\bel{le4}
$\displaystyle  \frac{1}{A+1} < \sqrt{\frac{a}{c}} < 
\log \frac{\alpha}{\beta}<  \sqrt{\frac{a}{b}} < 
\frac{1}{A-1}$.
\eel
\bep
From the mean value theorem one gets
\[
\log \alpha - \log\beta =\frac{s-r}{\sqrt{\xi ^2-4}}
\quad \mbox{for some $\xi$ satisfying $r < \xi < s$.}
\]
The claim follows after elementary computations, using
the explicit formulas for $s$ and $r$.
\eep

Equally simple computations yield the following.

\bel{le2} For  $AK \ge 34$ one has $\beta > 0.999\,  r$.
\eel

\bel{le5}
Let $\rho$ be a positive integer. Then for $AK\ge 2\rho +4$
one has 
\[
c-a \le  b+\frac{(2\rho +2)b}{\rho A}.
\]
\eel \bep 
The claim is equivalent to $\rho rA \le (\rho +1)b$, which,
on using the explicit formulas for $r$ and $b$, turns out
to be precisely $AK\ge 2\rho +4$.
\eep

\bel{le6}
Let $A_0\ge 2$, $K_0\ge 3$, and $\rho\ge 14$ be integers. If
$A\ge A_0$, $K\ge K_0$, and  $AK\ge 2\rho +4$ then 
\[
bc^2(c-a) <  0.992^{-1} \left( 1+\frac{ 2\rho +2}{\rho A_0} \right)
\left( \frac{1}{K_0}+\frac{1}{ 2\rho +2} \right)^4  \beta ^{8}.
\]
\eel
\bep
Notice that one has
\[
K^2bc=(r^2-4) (s^2-4) < r^2s^2=K^2r^4 \left( \frac{1}{K}+ \frac{1}{r}
\right)^2
\]
and, by the previous lemma,
\[
bc^2(c-a) \leq  \left( 1+\frac{ 2\rho +2}{\rho A_0} \right) b^2c^2 
< \left( 1+\frac{ 2\rho +2}{\rho A_0} \right)
\left( \frac{1}{K_0}+\frac{1}{ 2\rho +2} \right)^4  r^8,
\]
while Lemma~\ref{le2} yields
\[
\beta ^{8} > (0.999\,  r)^{8} > 0.992\, r^8.
\]
\eep

\bel{le7}
$\displaystyle \frac{\sqrt{bc}+\sqrt{ac}}{\sqrt{bc} - \sqrt{ab}}
< 1+\frac{5}{2A}$ if one of the following holds:
\[
\biggl\{\begin{array}{l}K=3,\\ A \ge 6,\end{array}
\quad \biggl\{\begin{array}{l}K=5,\\ A \ge 5,\end{array}
\quad \biggl\{\begin{array}{l}6 \le K \le 11,\\ A \ge 4,\end{array}
\quad \biggl\{\begin{array}{l}K \ge 12,\\ A \ge 3.\end{array}
\]
\eel
\bep
The desired inequality is equivalent to 
$2A\sqrt{ac}+(2A+5)\sqrt{ab} < 5\sqrt{bc}$.
Squaring this and replacing $\sqrt{bc}$ by the larger quantity 
$t$, we arrive at a bivariate polynomial inequality which is
easily seen to hold in each of the cases displayed above. 
\eep

By rewriting the linear form considered above in the form
\[
\Lambda = \log(\beta^{2\nu}\chi)-2m \log(\alpha/\beta),
\]
one may obtain a lower bound for $m$. 
\bel{le9} If $\nu = l-m$ with $m \ge 1$, then $m > (A-1)\nu \log\beta $.
\eel
\bep
Estimate (\ref{in:L}) implies
\[
-\alpha^{1-4m}+\log(\beta^{2\nu}\chi)<2m\log(\alpha/\beta)<\log(\beta^{2\nu}\chi).
\]
Since it is not difficult to check $\log \chi > \alpha^{1-4m}$, 
one has $m \log(\alpha/\beta)>\nu \log \beta$. 
The asserted inequality now follows from Lemma~\ref{le4}. 
\eep

%
\section{Auxiliary results for $\eps=2$}\label{sec2p2}

In this section we keep the notation
\[
\alpha =\frac{s+\sqrt{ac}}{2}, \quad \beta =\frac{r+\sqrt{ab}}{2},
\quad \chi =\frac{\sqrt{bc} + \sqrt{ac}}{\sqrt{bc} \pm \sqrt{ab}}.
\]

Results similar to those given in the previous section hold
for these algebraic numbers. The proofs contain no new ideas,
the differences appear in the numerical details. Therefore, we
avoid annoying repetitions by omitting the proofs.

\bel{le1p2}
$\alpha -\beta >K=s-r.$
\eel

\bel{le2p2}
One always has $\alpha > 0.998\, s$ and $\beta > 0.998\, r$.
Moreover,  for $AK\ge 30$ one has  $\alpha > 0.999\,  s$ and
$\beta > 0.999\,  r$.
\eel

\bel{le3p2}
For $AK\ge 43$ one has $\sqrt{ab}  > 0.999 \, r$,  $\sqrt{ac}
> 0.999\,  s$, and $0.999\, (\sqrt{ac} - \sqrt{ab}) < K$.
Moreover, for any $A\ge 23$ it holds $\sqrt{bc}  > 0.999 \, t$.
\eel

\bel{le4p2}
$\displaystyle  \frac{1}{A+1+2/K} < \log \frac{\alpha}{\beta}<
\frac{1}{A}$.
\eel

\bel{le5p2}
$\displaystyle  c-a < \left( 1+\frac{2}{A} \right)b$.
\eel

\bel{le6p2}
Let $A_0\ge 1$, $K_0\ge 1$, and $\rho\ge 15$ be integers. If
$A\ge A_0$, $K\ge K_0$, and  $AK\ge 2\rho $ then 
\[
bc^2(c-a) <  0.992^{-1} \left( 1+\frac{ 2}{ A_0} \right)
\left( \frac{1}{K_0}+\frac{1}{ 2\rho +2} \right)^4  \beta ^{8}.
\]
\eel

\bel{le7p2}
$\displaystyle \frac{\sqrt{bc}+\sqrt{ac}}{\sqrt{bc} - \sqrt{ab}}
< 1+\frac{5}{2A}$.
\eel

\bel{le9p2} 
If $\nu = l-m$ with $m \geq 1$, then $m > A \nu \log\beta $.
\eel

\section{Application of the hypergeometric method to the case $|\varepsilon|=2$}\label{sec:hg2}

The hypergeometric method is very effective when dealing with
small values of $A$. For the rest of the section we put
\[
 N=\frac{1}{2}(A^2+A)K+\varepsilon A, \quad \theta_1 =\sqrt{1-\frac{\varepsilon A}{N}},
\quad \theta_2 =\sqrt{1+\frac{\varepsilon}{N}}
\]
with $\varepsilon \in \{-2,2\}$. 

\bel{le11p2}  
Let $(x,y,z)$ be a  solution in positive integers
to the system of Diophantine equations~\eqref{eq:7} and \eqref{eq:8}. Then
\[
 \max \left\{  \left| \theta _1 -\frac{(A+1)x}{z}\right|,
 \left| \theta _2 -\frac{(A+1)y}{Az}\right|
\right\} <2(A+1)(A+1+2\cdot K^{-1})z^{-2}.
\]
\eel
\bep
Follow the proof of Lemma~6 from~\cite{fht} with a twist on the
final step --- use $A+1+2\cdot K^{-1}$ instead of $A+3$ as an 
upper bound for $\sqrt{c/a}$.
\eep

A lower bound for the left side of the inequality in the previous 
lemma can be obtained by using results on simultaneous approximations 
of algebraic numbers which are close to 1.

As already mentioned, we study small values of $A$ with
the help of the hypergeometric method. The next result
contains the outcome of the study.

\begin{prop}\label{prop:hg2}
Let $a=K$, $b=A^2K+2\varepsilon A$, $c=(A+1)^2K+2\varepsilon(A+1)$ with 
$\varepsilon \in \{-2,2\}$ and positive integers $A$, $K$. 
Suppose that $\{a,b,c,d\}$ is a $D(4)$-quadruple with $d>2$ not 
given by~\eqref{d}. 
If $A \ge A_0$, then $K<240.24(A+1)+K_0$, where
\[
(A_0,K_0) \in \{(1326,0),(454,1000),(3,23000),(2,210000)\}.
\]
\end{prop}
\begin{proof}
Suppose that $K \geq 240.24(A+1)$. 
On applying Lemma~\ref{le11p2} and Theorem~\ref{thm:RB} with $p_1=A(A+1)x$, 
$p_2=(A+1)y$, $q=Az$, $N=(A^2+A)K/2+\varepsilon A$, one gets
\beq{in:hg2}
z^{2-\lambda} < 2C^{-1} A^{\lambda} (A+1)(A+1+2\cdot K^{-1}),
\eeq 
where $C^{-1}=2.838\cdot 10^{28}(A+1)N$.  
It is easy to see from the proof of Lemma 5 in \cite{fht} that 
\beq{in:z>m}
\log z > 2m \log((A+1)K+\varepsilon-2).
\eeq
The assumption $K \ge 240.24(A+1)$ ensures $\lambda<2$, which, 
combined  with Lemmas \ref{le9}, \ref{le9p2} and
inequalities (\ref{in:hg2}), (\ref{in:z>m}), implies
\begin{align}\label{in:KA}
(A-1)\nu \log \beta<\frac{\log(2C^{-1}A^2(A+1)(A+1+2/K))}{2(2-\lambda)\log((A+1)K+\varepsilon-2)}.
\end{align}
Since
\[
2-\lambda=\frac{\log\left(\frac{0.669\, N}{80A(A+1)^2}\right)}{\log\left(\frac{0.669\, N^2}{4A(A+1)}\right)}
         =\frac{\log\left(\frac{0.669\, \{(A+1)K+\varepsilon-2\}}{160(A+1)^2}\right)}{\log\left(\frac{0.669\, A\{(A+1)K+\varepsilon-2\}^2}{16(A+1)}\right)},
\]
the right-hand side of (\ref{in:KA}) is a decreasing function of $K$. 
Therefore, one can easily verify the assertion by using (\ref{in:KA}) 
with $\nu \ge 1$ and a computer.
\end{proof}

\section{Application of Baker's method to the case $|\varepsilon|=2$}\label{sec:baker2}
\begin{prop}\label{prop:2log}
Let $a=K$, $b=A^2K+2\varepsilon A$, $c=(A+1)^2K+2\varepsilon(A+1)$ with 
$\varepsilon \in \{-2,2\}$ and positive integers $A$, $K$. 
Suppose that $\{a,b,c,d\}$ is a $D(4)$-quadruple with $d>2$ not 
given by~\eqref{d}. 
Then, we have 
\[
A \le \begin{cases} 2800 & \text{if $\varepsilon=-2;$}\\
                    3365 & \text{if $\varepsilon=2$.}
      \end{cases}
\]
\end{prop}
\begin{proof}
Recall that
\[
\alpha = \frac{s+\sqrt{ac}}{ 2}, \quad \beta=\frac{r+\sqrt{ab}}{ 2}, 
\quad \chi = \frac{\sqrt{bc} + \sqrt{ac}}{ \sqrt{bc}\pm \sqrt{ab}}.
\]
All these algebraic numbers belong to the number field (of degree four) 
${\bQ}(\sqrt{ab},\sqrt{ac})$, whose ${\bQ}$-automor\-phisms are defined by
$(\sqrt{ab},\sqrt{ac})\mapsto (e_1 \sqrt{ab},e_2 \sqrt{ac})$, 
where $e_1$, $e_2\in \{-1,+1\}$. It follows that the conjugates of 
$\chi $ are $\chi $ and
\[
\chi ' = \frac{\sqrt{bc} + \sqrt{ac}}{\sqrt{bc}\mp \sqrt{ab}}, 
\quad \chi '' = \frac{\sqrt{bc} - \sqrt{ac}}{ \sqrt{bc}\pm \sqrt{ab}}, 
\quad \chi ''' = \frac{\sqrt{bc} - \sqrt{ac}}{\sqrt{bc}\mp \sqrt{ab}}.
\]
Hence
\[
0 < \chi  \chi ''' = \chi ' \chi '' = \frac{bc -ac}{ bc -ab} <1
\]
and
\[
0 <  \chi , \, \chi ', \, \chi '', \, \chi ''', \, (1-\chi)(1-\chi ') 
, \,  (1-\chi '')(1-\chi ''').
\]
This shows that $(bc-ac)^2$ is a denominator for $\chi $ and that
\[
\hh(\chi ) \le \frac{1}{ 4} \left( \log(b^2(c-a)^2) + \log%
\frac{c(\sqrt a +\sqrt b)^2}{ b(c-a)}\right)
 = \frac{1}{ 4} \log \bigl(b c (c-a) (\sqrt a +\sqrt b)^2 \bigr).
\]
Here  $c=a+b + 2r$, so that $c > (\sqrt a +\sqrt b)^2$ and 
\[
\hh(\chi ) \le   \frac{1}{ 4} \log \bigl(b c^2(c-a)  \bigr).
\]
Now we see that, for $A\ge 80$, Lemmas \ref{le6} and \ref{le6p2} yield
\[
\hh(\chi ) <  \begin{cases}0.014 \log \beta & \text{if $\varepsilon=-2$},\\
                           2.005 \log \beta & \text{if $\varepsilon=2$}.
              \end{cases}
\]
The conjugates of $\alpha/\beta$ are $\alpha/\beta$ and
\[
\frac{s -\sqrt {ac}}{ r +\sqrt{ab}}, \quad 
\frac{s +\sqrt {ac}}{ r -\sqrt{ab}}, \quad 
\frac{s -\sqrt {ac}}{r -\sqrt{ab}}.
\]
As among these four numbers only the first and the third ones are of 
modulus greater than $1$, it easily follows that
\[
\hh(\alpha /\beta)= \frac{1}{2} \log \alpha,
\]
because $\alpha$ and $\beta$ are algebraic units. Moreover since 
$\chi $ is obviously not a unit, the numbers $\beta^{2\nu} \chi $ 
and $\alpha/\beta$ are multiplicatively independent. Now we are 
ready to apply Laurent's lower bounds~\cite{lau} to the linear form
\[
\Lambda = \log (\beta^{2\nu} \chi ) - 2m \log(\alpha/\beta).
\]
With the notation of~\cite{lau} we have
\[
b_1=2m, \quad b_2=1, \quad \alpha_1 = \alpha/\beta, \quad \alpha_2= \beta^{2\nu} \chi .
\]
Using the above study, and the inequality $\log \alpha_1 >1/(A+1+2/K)$ 
following from Lemmas \ref{le4} and \ref{le4p2}, one can choose
\[
a_1 \ge 4 \log \alpha +\frac{\rho-1}{A+1+2/K} 
\]
and, in view of Lemmas \ref{le7} and \ref{le7p2}, the choice
\[
a_2 \ge \bigl( 2\nu (\rho+3)+q_2\bigr) \log \beta + (\rho-1)%
\log \left( 1+\frac{5}{ 2A}\right)
\]
is legitimate for $A\ge 80$, where $q_2=0.112$ or $16.04$ depending 
on $\varepsilon=-2$ or $2$, respectively.

By way of illustration, we present the details in case
$\rho =37$, $\mu =0.63$. We shall also suppose that $A \ge 2700$.
Then we may take
\begin{align*}
a_1 & = 4.0017 \log \alpha, \\
a_2 & = (80\nu +q_2') \log \beta ,
\end{align*}
where $q_2'=0.116$ or $16.045$ depending on $\varepsilon=-2$ or $2$, respectively. 
From $\alpha>\beta$ we then get
\begin{align*}
\frac{b_1}{a_2}+ \frac{b_2}{a_1} & =
\frac{2m}{(80\nu +q_2') \log \beta} +
\frac{1}{4.0017 \log \alpha} \\
& <  \frac{m+10\nu +q_2'/8}{(40\nu + q_2'/2) \log \beta},
\end{align*}
which implies that
\[
h  = 4\log \left( \frac{m+10\nu + q_2'/8}{(40\nu + q_2'/2)\log \beta} \right) +11.913
\]
satisfies the hypotheses of Laurent's theorem.

Suppose $h \le 28.9$. 
If $\varepsilon=-2$, then it results from Lemma~\ref{le9}
\[
(A-1)\nu \log\beta < m <  (40\nu + 0.058) \exp (4.24675)
 \log \beta,
\]
that is, 
\begin{align}\label{A<:-2}
 A < \left(40 + \frac{0.058}{\nu}\right)\exp(4.24675)+1.
\end{align}
Similarly, if $\varepsilon=2$, then Lemma~\ref{le9p2} implies
\begin{align}\label{A<:2}
A< \left(40+ \frac{8.0225}{\nu}\right)\exp(4.24675).
\end{align}

Suppose $h>28.9$. 
Combining inequality (\ref{in:L}) with Theorem~2 from~\cite{lau} yields
\begin{align}\label{in:Laurent}
(4m-1)\log \alpha<
C\left(h+\frac{\lambda}{\sigma}\right)^2a_1a_2+\sqrt{\omega \theta}\left(h+\frac{\lambda}{\sigma}\right)+\log\left(C'\left(h+\frac{\lambda}{\sigma}\right)^2a_1a_2\right),
\end{align}
where 
\begin{align*}
\sigma&=\frac{1+2\mu-\mu^2}{2},\quad \lambda=\sigma \log \rho,\\
\omega&=2\left(1+\sqrt{1+\frac{1}{4H^2}}\right),\quad \theta=\sqrt{1+\frac{1}{4H^2}}+\frac{1}{2H},\\
h & \ge \max \left\{4\left(\log \left(\frac{b_1}{a_2}+\frac{b_2}{a_1}\right)+\log \lambda+1.75\right)+0.06,~\lambda,~2\log2\right\},\\
H&=\frac{h}{\lambda}+\frac{1}{\sigma},\\
C&=\frac{\mu}{\lambda^3\sigma}\left(\frac{\omega}{6}+\frac12\sqrt{\frac{\omega^2}{9}+\frac{8\lambda \omega^{5/4}\theta^{1/4}}{3\sqrt{a_1a_2}H^{1/2}}+\frac43\left(\frac{1}{a_1}+\frac{1}{a_2}\right)\frac{\lambda \omega}{H}}\,\right)^2,\\
C'&=\sqrt{\frac{C\sigma \omega \theta}{\lambda^3\mu}}.
\end{align*}
If $\varepsilon=-2$, then inequality (\ref{in:Laurent}) shows that 
\[
\frac{m}{(40\nu+0.058)\log \beta}<69.799,
\]
which together with Lemma \ref{le9} implies 
\begin{align}\label{A<:-2'}
A < 69.799\left(40+\frac{0.058}{\nu}\right)+1.
\end{align}
Inequalities (\ref{A<:-2}) and (\ref{A<:-2'}) together yield $A \le 2800$ for all $\nu \ge 1$. 
If $\varepsilon=2$, then inequality (\ref{in:Laurent}) 
and Lemma \ref{le9p2} together show that 
\begin{align}\label{A<:2'}
A < 70.073\left(40+\frac{8.0225}{\nu}\right),
\end{align}
with which combining (\ref{A<:2}) implies $A \le 3365$ for all $\nu \ge 1$. 

\end{proof}
\section{Proof of Theorem $\ref{thm:main}$}\label{sec:pr}
Although there remain only finitely many cases to check, 
we will try to make the number as small as possible in order to save a computation time. 
\begin{lem}\label{lem:nu}
Suppose $V_{2l}=W_{2m}$ holds for some integers $l$ and $m$ with $m \ge 2$. 
If $\nu=l-m$, then $\nu \ge 11$. 
\end{lem}
\begin{proof}

Remark that the integer $m$ is completely determined for fixed $A$, 
$K$ and $\nu$. In fact, $m$ is expressed as
\[
m=\frac{\nu \log \beta+0.5 \log \chi}{\log(\alpha/\beta)}
 - \frac{\Lambda}{2\log(\alpha/\beta)},
\]
where the term after the minus sign is positive and less than $1$ in 
view of (\ref{in:L}). Thus we have
\[
m=\left\lfloor \frac{\nu \log \beta + 0.5 \log \chi}{\log(\alpha/\beta)} \right\rfloor.
\]
For each set of values of $A$, $K$ bounded as in Propositions 
\ref{prop:hg2} and \ref{prop:2log}, and for each  $\nu$ with 
$1 \le \nu \le 10$ we computed the linear form  $\Lambda$ 
and found that $\Lambda > \alpha^{1-4m}$, which contradicts (\ref{in:L}). 
Our computer needed about $30$ hours to perform these computations. 
\end{proof}
\begin{prop}\label{prop:hg2'}
Keep the hypotheses of Proposition $\ref{prop:hg2}$. Then: 

{\rm (1)} $K < 240.24\,(A+1)+740$. 
Moreover, if $40 \le A \le 2810$, then $K < 237.05\,(A+1)$.\par
{\rm (2)} $A \le 2796$ if $\varepsilon=-2$ and $A \le 2810$ if $\varepsilon=2$. 
\end{prop}
\begin{proof}
(1) Inequality (\ref{in:KA}) with $A \ge 2$ and $\nu \ge 11$ shows 
the first assertion. In a way similar to Lemma \ref{lem:nu}, one can 
check by computer that if $A \le A_0$, then $\nu \ge \nu_0$, where 
\[
(A_0,\nu_0) \in \{(900,12),~(360,14),~(40,25)\},
\]
and show the following: \\[6pt]
Substituting $K=237.05\,(A+1)$ and each value of $A$ in each of 
the ranges $40 \le A <360$, $360 \le A < 900$, $900 \le A \le 2810$ 
into the quantities $p$,~$P$,~$l$,~$L$ defined in Section~\ref{R} 
immediately after (\ref{in:pPlL}),   
Lemma~\ref{lem:RB} yields renewed $C^{-1}$ and $\lambda$, 
that are not compatible with inequality (\ref{in:KA}).\\[6pt]
Thus one obtains the revised bound $K < 237.05\,(A+1)$ for 
$40 \le A \le 2810$. \par
(2) Inequalities (\ref{A<:-2}), (\ref{A<:-2'}) together with 
$\nu \ge 11$ give the asserted inequality for $\varepsilon=-2$. 
When $\varepsilon=2$ one verifies that $\nu \ge 70$ for $A>2810$ 
(this takes only a few hours of computer time), which together 
with inequalities (\ref{A<:2}), (\ref{A<:2'}) implies the result. 
\end{proof}

In order to get an absolute upper bound for $m$, we appeal to 
Matveev's theorem for three logarithms. 

\begin{thm}[\cite{Matveev}]\label{thm:Mat}
Let $\lambda_1$, $\lambda_2$, $\lambda_3$ be $\bQ$--linearly 
independent logarithms of non-zero algebraic numbers and let 
$b_1$, $b_2$, $b_3$ be rational integers with $b_1\not=0$.
Define $\alpha_j=\exp(\lambda_j)$ for $j=1,~ 2,~ 3$ and

\[
    \Lambda = b_1 \lambda_1 + b_2 \lambda_2 + b_3 \lambda_3 .
\]
Let $D$ be
the degree of the number field $\bQ(\alpha_1,\alpha_2,\alpha_3)$ over $\bQ$. Put
\[
    \chi = [\bR(\alpha_1,\alpha_2,\alpha_3) : \bR].
\]
Let $ A_1$, $A_2$, $A_3$
be positive real numbers, which satisfy
\[
    A_j \ge \max\bigl\{D \hh(\alpha_j),  |\lambda_j|, 0{.}16 \bigr\}
    \quad (1\le j\le 3) .
\]
Assume that
\[
    B \ge \max\Bigl\{1,
    \max \bigl\{ |b_j|A_j/A_1; \, 1\le j \le 3 \bigr\}  \Bigr\}.
\]
Define also
\[
    C_1 = \frac{5\times 16^5}{6\chi}\,e^3\,(7+2\chi)\left(\frac{3e}{2}\right)^\chi\!
    \Bigl( 20{.}2+\log\bigl(3^{5{.}5} D^2\log (e D )\bigr)\Bigr).
\]
Then
\[
    \log |\Lambda| > - C_1\,  D^2\,  A_1\, A_2 A_3\,
    \log\, \bigl( 1{.}5\,e D B \log(e D )  \bigr) .
\]
\end{thm}

In our case we choose 
\[
\alpha_1=\chi,\ b_1=1, \quad \alpha_2=\beta, \ b_2=2\nu, \quad 
\alpha_3=\alpha/\beta, \ b_3=-2m.
\]
Then we can take, for $j=1$, $2$, $3$,
\[
A_j = 4 \hh(\alpha_j)
\]
and
\[
B = 2m A_3/A_1.
\]
With these values we get
\[
m < 3.4 \cdot 10^{16}.
\]

It now remains only to perform the reduction procedure. 
Let 
\[
\alpha=\frac{s+\sqrt{ac}}{2},\quad \gamma=\frac{t+\sqrt{bc}}{2},
\quad \mu=\frac{\sqrt{b}(\sqrt{c} \pm \sqrt{a})}{\sqrt{a}(\sqrt{c}\pm \sqrt{b})},
\]
where the signs coincide. 
If $z=v_{2m}=w_{2n}$ has a solution with $mn \neq 0$, 
then the linear form $\Omega=2m \log \alpha -2n \log \gamma +\log \mu$ 
satisfies
\begin{align*}
0<\Omega < 2ac \,\alpha^{-4m}
\end{align*}
(cf.~\cite[Section 4]{fht}). 
The following is a version of the Baker-Davenport lemma 
(\cite[Lemma]{Baker-Davenport:1969}), due to Dujella and Peth\H{o}, 
needed here. 
\begin{lem}{\rm (\cite[Lemma 5 a\small{)}]{dup})}\label{lem:red}\quad 
Let $M$ be a positive integer, and $\kappa$, $\xi$ real numbers. 
Let $P/Q$ be a convergent of the continued fraction expansion of 
$\kappa$ such that $Q>6M$. Put $\eta=||\xi Q||-M\,||\kappa Q||$, 
where $||\cdot||$ denotes the distance from the nearest integer. 
If $\eta>0$, then there exists no solution of the inequality
\[
0 < m \kappa -n + \xi < EB^{-m}
\]
in integers $m$ and $n$ with
\[
\frac{\log(EQ/\eta)}{\log E}\leq m < M.
\]
\end{lem}
We apply Lemma \ref{lem:red} with 
\[
\kappa=\frac{\log \alpha}{\log \gamma},\quad \xi=\frac{\log \mu}{2\log \gamma},
\quad E=\frac{ac}{\log \gamma},\quad B=\alpha^4 
\]
and $M=3.4\cdot 10^{16}$ for $A$, $K$ satisfying
\begin{align*}
\begin{cases}
K<237.05\,(A+1) & \text{if $40 \le A \le 2810$},\\
K<240.24\,(A+1)+740 & \text{if $2 \le A \le 39$}.
\end{cases}
\end{align*}
The computation was carried out by running a program developed in 
PARI/GP (\cite{gp}) with the precision
\[
{\rm realprecision} = \max\{180, 10 \lceil A/100 \rceil\},
\]
and no counter-example was found. 
The verification took around three months. 
This completes the proof of Theorem \ref{thm:main}. 

\begin{proof}[Proof of Corollary $\ref{cor:main}$]
Suppose that $r \equiv \varepsilon \pmod{a}$, and put $r=ka+\varepsilon$ 
with an integer $k$. Then, $b=k^2a+2\varepsilon k$ and 
$c=(k+1)^2a+2\varepsilon(k+1)$. Applying Theorem \ref{thm:main} 
to the triple $\{a,b,c\}$ with $K=a$ and $A=k$, one can obtain 
the first assertion. The second assertion is an immediate consequence 
of the first one together with the fact that  one always has
$r^2 \equiv \varepsilon^2 \pmod{a}$. 
\end{proof}

\section{An application to the study of $D(1)$-quintuples} \label{se7}

In this section we consider a hypothetical  $D(1)$-quintuple
 $\{a,b,c,d,e\}$ with $a<b<c<d<e$,  $c=a+b+2r$,  and $r=a^2-\Delta$.
Theorem~\ref{thm:main} ensures $\Delta >1$. An upper bound of the 
type $\Delta < a^2-a$ is derived from the obvious inequality $a < r$.
Our considerations are based on a recent result, recalled here for
reader's convenience.
 
\bel{lem:t13} (\cite[Theorem~1.3]{CFF})
Let  $\{a,b,c,d,e\}$ be a  quintuple  with $a<b<c<d<e$  and  
$c=a+b+2\sqrt{ab+1}$. Then $b < a^3$ and $\gcd(b,c)=1$.
In particular, at least one of $a$, $b$ is odd.
\eel

In conjunction with Lemma~3.4 from~\cite{CFF}, which essentially
says that for each $D(1)$-quintuple one has $b > 4000$, 
Lemma~\ref{lem:t13} gives the lower bound $a \ge 16$. From 
$ab+1=r^2$ one obtains $b=a^3-2a\Delta +(\Delta ^2-1)/a$,
whence the conclusion that the integer $a$ is a divisor greater
than 15 for $\Delta ^2-1$. This in turn implies $\Delta \ge 5$.

When $\Delta = 5$, the only admissible divisor of 24 is $a=24$,
so that $b=13585=5\cdot 11\cdot 13\cdot 19$ and $c=14751=11\cdot 1341$.
Then $\gcd (b,c)=11$, in contradiction with Lemma~\ref{lem:t13}.

Up to now we have proved that $\Delta \ge 6$, an information with
striking consequences.

\bepr{pr:app} In the hypothesis of  Lemma~\ref{lem:t13} one has
$b<a^3-11a$ and $a\ge 20$.
\eepr \bep 
The first assertion follows from 
$a^3-2a\Delta +(\Delta ^2-1)/a < a^3-11a$,
which is equivalent to $\Delta^2 -2a^2\Delta +11a^2 \le 0$ and
to $a^2-\sqrt{a^4-11a^2} \le \Delta  \le a^2+\sqrt{a^4-11a^2}$.
The right inequality is much weaker than  $\Delta < a^2-a$, while 
the left one is easily derived by interlacing 6 between its terms.

The second assertion in the conclusion follows from 
Corollary \ref{cor:main}, since $17$ and $19$ are  prime numbers,
while $18$ is twice a power of a prime.
\eep

If so needed/wanted, one can pursue the analysis and eliminate
other values of $\Delta$. For instance, when $\Delta =6$,
$\Delta  ^2-1$ has unique divisor greater than 16, namely $a=35$.
To conclude that $\Delta > 6$ one has to prove that the $D(1)$-triple
$(35,42456,44929)$ has unique extension to a $D(1)$-quadruple.
When  $\Delta =7$, the only admissible candidates for the smallest
entry are $a=24$ and $a=48$. The former value entails $b=13490$, so
that $\gcd (b,c)=2$, which means that in this case one can not obtain
a $D(1)$-quintuple, so it remains to study the extendability of
the triple $(48,109921,114563)$. In order to prove that one has
$\Delta > 10$, three more triples, viz., $(a,b,c)=(21,8928,9815)$, 
$(80,510561,523423)$, $(99,968320,988001)$, need to be shown to have
unique extension to a $D(1)$-quadruple. 

Each lower bound $\Delta \ge \Delta _0$ can be used to improve
upon Proposition~\ref{pr:app}.

\medskip

Another kind of upper bounds for $b$ can be obtained from $\Delta>1$ 
and $\Delta^2 \equiv 1 \pmod{a}$. In other words, it holds
\[
\Delta \ge \sqrt{a+1}\quad \text{and}\quad b \le a^3-2a\,\sqrt{a+1}+1. 
\]
The extremal case $\Delta=\sqrt{a+1}$ frequently appears in the
observation above, which motivated us to show the following.

\begin{prop}\label{prop:delta}
Let $\Delta \ge 6$ be an integer and
\[
a=\Delta^2-1, \quad  b=\Delta^6-3\Delta^4-2\Delta^3+3\Delta^2+2\Delta, 
\quad c=\Delta^6-\Delta^4-2\Delta^3+1. 
\]
If $\{a,b,c,d\}$ is a $D(1)$-quadruple,  
then
\begin{align*}
d=d_+&=4\Delta^{14}-20\Delta^{12}-16\Delta^{11}+40\Delta^{10}+56\Delta^9-16\Delta^8\\
&\quad-72\Delta^7-32\Delta^6+24\Delta^5+32\Delta^4+8\Delta^3-4\Delta^2-4\Delta.
\end{align*}
\end{prop}
\begin{proof}
We easily verify that
\[
\ r=\Delta^4-2 \Delta^2-\Delta+1, \quad s=\Delta^4-\Delta^2-\Delta, 
\quad t=\Delta^6-2\Delta^4-2\Delta^3+\Delta^2+\Delta+1.
\]
Suppose that $\{a,b,c,d\}$ is a $D(1)$-quadruple with $d>d_+$. 
Putting $ad+1=x^2$, $bd+1=y^2$, $cd+1=z^2$, and eliminating $d$ from 
these equations, we obtain the following system of Pellian equations:
\begin{align}
az^2-cx^2=a-c, \label{D:zx}\\
ay^2-bx^2=a-b. \label{D:yx}
\end{align}
Since $c=a+b+2r$, the same argument as Section 2 in \cite{HT2} 
applies and one finds that any solution to the system of Pellian 
equations (\ref{D:zx}), (\ref{D:yx}) is given by $x=W_{2m}=V_{2l}$, 
where
\begin{align*}
W_0=1,\quad &W_1=s \pm a,\quad W_{m+2}=2sW_{m+1}-W_m,\\
V_0=1,\quad &V_1=r+a,\quad V_{l+2}=2rV_{l+1}-V_l.
\end{align*}
Put
\[
\alpha =s+\sqrt{ac}, \quad \beta =r+\sqrt{ab},
\quad \chi =\frac{\sqrt{bc} + \sqrt{ac}}{\sqrt{bc} \pm \sqrt{ab}}.
\]

\medskip

The following results are the analogs of the preceding ones.

\medskip

\textbf{Lemma 3.5'.} \quad  $\alpha -\beta >2a=2(s-r).$

\medskip

\textbf{Lemma 3.6'.} \quad   $1/ \Delta^2< \sqrt{ {a/c}} < 
\log(\alpha/\beta) < \sqrt{ {a/b}} < {1 /(\Delta^2-2)}$.

\medskip

\textbf{Lemma 3.7'.} \quad  $\beta > 1.9999\, r$.

\medskip

\textbf{Lemma 3.8'.} \quad $  c-a \le b+ (2\Delta +3)b /( \Delta^3-\Delta)$.

\medskip

\textbf{Lemma 3.9'.} \quad  \textit{If} $\Delta\ge \Delta_0\ge 6$ 
 \textit{then}
\[
 bc^2(c-a) < 1.9999^{-8} \left( 1+\frac{ 2\Delta_0 +3}{\Delta_0^3-\Delta_0} \right) 
 \frac{\beta ^8}{(\Delta_0^2-2)^4} .
\]

\medskip

\textbf{Lemma 3.11'.} \quad
$\displaystyle  \frac{\sqrt{bc}+\sqrt{ac}}{\sqrt{bc} - \sqrt{ab}} < 
\frac{\sqrt{bc}+\sqrt{ac}}{\sqrt{bc} + \sqrt{ab}} < 1+\frac{1}{\Delta^2-2}$.

\medskip

\textbf{Lemma 3.12'.} \quad  \textit{Consider the linear form}
\[
\Lambda = \log(\beta^{2\nu}\chi)-2m \log(\alpha/\beta),
\]
 \textit{and put} $\nu = l-m$  \textit{with} $m \ge 1$. 
 \textit{Then} $m > (\Delta^2-2)\nu \log\beta $.

\medskip

We get again
\def\hh{{\rm h}}
\[
\hh(\chi ) \le \frac{1}{4} \left( \log(b^2(c-a)^2) +
 \log \frac{c(\sqrt a +\sqrt b)^2}{b(c-a)}\right)
\le   \frac{1}{4} \log \bigl(b c^2(c-a)  \bigr),
\]
hence
\[
\hh(\chi ) < 2 \log \beta
\]
by Lemma 3.9'. Again
\[
\hh(\alpha /\beta)=\frac{1}{2} \log \alpha.
\]

As the  numbers $\beta^{2\nu} \chi $ and $\alpha/\beta$  are 
multiplicatively independent over $\bQ$, we can apply Laurent's 
lower bounds~\cite{lau} to the linear form
\[
\Lambda = \log (\beta^{2\nu} \chi ) - 2m \log(\alpha/\beta).
\]
With the notation of this paper we have
\[
b_1=2m, \quad b_2=1, \quad \alpha_1 = \alpha/\beta, \quad 
\alpha_2= \beta^{2\nu} \chi .
\]

Using the above study and the inequality 
$\log \alpha_1 < 1/(\Delta^2-2)$ established in Lemma~3.6', 
one can choose
\[
a_1 \ge 4 \log \alpha + \frac{\rho-1}{\Delta^2-2}.
\]
Moreover,  the choice 
\[
a_2 \ge 2\bigl( \nu (\rho+3)+8\bigr) \log \beta + (\rho-1)
\log \left( 1+ \frac{1}{\Delta^2-2}\right)
\]
is legitimate  by Lemma 3.10'.


Now we suppose $\Delta >60$. We omit the details since the previous 
study applies almost word for word after the substitution 
$A \mapsto \Delta^2-1$.
Laurent's estimates lead to a contradiction. We conclude that
\[
\Delta \le 60.
\]


Then we can apply Matveev's estimates to the (expanded) linear form 
in three logarithms
\[
\Lambda = \log \chi + 2 \nu  \log \beta  - 2m \log(\alpha/\beta)
\]
and we get
\[
m < 10^{17}.
\]
To end the proof we use the Baker-Davenport lemma and a computer 
(with a real precision of 200 digits). The verification took less 
than 1 second. 
\end{proof}

On noting that, by Lemma~\ref{lem:t13}, $\Delta^2-1$ and $a$ must 
be divisible by exactly the same power of $2$ when $a$ is even, 
from Proposition~\ref{prop:delta} one deduces $\Delta \ge \sqrt{3a+1}$.
This in turn readily implies the first claim in the conclusion
of  Proposition~\ref{pr:appfin}. The lower bound on $a$ has been 
obtained by performing the reduction procedure for $a=21$, $24$. 
Improved versions are easily available after similar computations 
for values $a$  either divisible by 8 or odd and not excluded by 
Corollary~\ref{cor:main}.

Similar considerations lead to Proposition~\ref{pr:appsup}.
Following is a sketch of the ideas involved in its proof.

Trudgian has combined results from~\cite{eff},~\cite{eu}, 
and~\cite{CFF} to show in~\cite{tim} that in any $D(1)$-quintuple 
whose second smallest element is less than four times the smallest
one, the smallest three elements form a regular triple. With the
notation fixed in this section, we therefore have
\[
 r=2a-\delta
\]
for some positive integer $\delta$ that has to be odd by 
Lemma~\ref{lem:t13} above and Theorem~1.2 from~\cite{CFF},
which says that if both $a$ and $b$ are odd then $b > 40a/9$.

Theorem~\ref{thm:main} ensures $\delta >1$.
From
\[
 b=4a-4\delta +\frac{\delta ^2-1}{a}
\]
we get that $a$ divides the positive integer $\delta ^2-1$,
so that $\delta \ge \sqrt{a+1}$. As before we conclude that
$a$ and  $\delta ^2-1$ have the same 2-adic valuation when
$a$ is even. Since, on the one hand, $a= \delta ^2-1$ is 
tantamount to $b=4\delta ^2-4\delta -3$ and, on the other 
hand, routine computations show that the triple 
$(a,b,c)= (\delta ^2-1,4\delta ^2-4\delta -3, 9\delta ^2-6\delta -8)$ 
can not be prolongated to a $D(1)$-quintuple, it results 
$\delta ^2-1 \ge 3a$. Hence,
\[
 b\le 4a -4\sqrt{3a+1} +3.
\]
The lower bound on $a$ follows from this and the complementary
inequality $b> 130000$, taking into account that an even value
of $a$ must be divisible by 8.

\section{Concluding remarks}\label{sec:conc}
In this paper we completed the work of previous authors and proved
in Theorem~\ref{thm:main} that each triple in the four families has
unique extension to a quadruple. It is for the first time in the 
literature that the extendability of a two-parameter family is
unconditionally settled.                         

The proof illustrates the known empirical fact that while the 
existence of `small' or `big' solutions can be relatively easily
decided, it is much more difficult to treat solutions of `medium
size'. Our attempt was successful due to use of linear forms in
the logarithms of two algebraic integers. One critical aspect of
such an approach is the need for sharp bounds for the difference
of integer coefficients of the logarithms. In the present study
we got such an information in Lemmas~\ref{le9} and~\ref{le9p2}.
It remains for future works to obtain similar bounds for general 
triples, not necessarily given parametrically.

As mentioned several times, the triples considered in this article 
are regular in the sense that $c=a+b+2r$. Another interesting 
direction for future work is to deal with non-regular triples.


\end{document}